\documentclass{amsart}

\usepackage{amssymb}
\usepackage{amsmath}
\usepackage{amsthm}
\usepackage{amsfonts}

\usepackage{a4wide}
\usepackage{longtable}
\usepackage{multirow}
\usepackage{setspace}

\newtheorem{proposition}{Proposition}[section]
\newtheorem{theorem}{Theorem}[section]
\newtheorem{corollary}{Corollary}[section]

\theoremstyle{definition}
\newtheorem{definition}{Definition}[section]

\newtheorem{remark}{Remark}[section]

\newtheorem{example}{Example}[section]

\newtheorem{conjecture}{Conjecture}[section]

\numberwithin{equation}{section}

\begin{document}


\title[Harmonically balanced capitulation over quadratic fields of type (9,9)]{Harmonically balanced capitulation \\ over quadratic fields of type (9,9)}

\author{Daniel C. Mayer}
\address{Naglergasse 53\\8010 Graz\\Austria}
\email{algebraic.number.theory@algebra.at}
\urladdr{http://www.algebra.at}

\thanks{Research supported by the Austrian Science Fund (FWF): project P 26008-N25}

\subjclass[2000]{Primary 20D15, 20E18, 20E22, 20F05, 20F12, 20F14, 20--04; secondary 11R37, 11R11, 11R29, 11Y40}

\keywords{Finite \(3\)-groups, Artin transfers to low index subgroups,
Artin pattern, transfer kernels and targets, derived length, relation rank,
searching strategy via \(p\)-group generation algorithm;
first and second Hilbert \(3\)-class field, maximal unramified pro-\(3\) extension,
Shafarevich cohomology criterion, 
\(3\)-tower length, quadratic fields, structure of \(3\)-class groups,
capitulation of \(3\)-classes in unramified \(3\)-extensions, capitulation types}

\date{August 06, 2019}


\begin{abstract}
The isomorphism type of the Galois group \(\mathfrak{G}\) of finite \(3\)-class field towers
of quadratic number fields with \(3\)-class group of type \((9,9)\)
is determined by means of Artin patterns
which contain information on the transfer of \(3\)-classes to unramified abelian \(3\)-extensions.
First, as an approximation of the group \(\mathfrak{G}\),
its metabelianization \(M=\mathfrak{G}/\mathfrak{G}^{\prime\prime}\),
which is isomorphic to the Galois group of the second Hilbert \(3\)-class field,
is sought by sifting the SmallGroups library with the aid of pattern recognition.
In cases with order \(\lvert M\rvert>3^8\), the SmallGroups database must be extended
by means of the \(p\)-group generation algorithm,
which reveals new phenomena of groups with harmonically balanced transfer kernels
and trees with periodic trifurcations.
Bounds for the relation rank \(d_2(M)\) of \(M\)
in dependence on the signature of the quadratic base field
admit the decision whether the derived length of \(\mathfrak{G}\) is
\(\mathrm{dl}(\mathfrak{G})=2\) or \(\mathrm{dl}(\mathfrak{G})\ge 3\).
\end{abstract}

\maketitle


\section{Introduction}
\label{s:Intro}
\noindent
The investigation of \(3\)-class field towers of quadratic fields \(K\)
with \(3\)-class group \(\mathrm{Cl}_3(K)\simeq C_9\times C_9\)
was suggested by Arnold Scholz in a letter of 12 October 1930 to Helmut Hasse
\cite[pp. 155--156]{LmRq}.
Due to limited computational facilities at that time,
Scholz could not find suitable base fields \(K\),
and consequently his project remained unfinished up to now.
In collaboration with M. F. Newman
\cite{MaNm},
we intentionally prepared the revival of the Scholz project
by studying the analogous case of
\(2\)-class field towers of quadratic fields \(K\)
with \(2\)-class group \(\mathrm{Cl}_2(K)\simeq C_4\times C_4\),
inspired by research results of Benjamin and Snyder
\cite{BeSn}.
On the one hand, the \(2\)-power extensions are computationally easier
because of their modest degrees, but on the other hand,
the broader variability of their transfer kernels discourages
the occurrence of \textit{harmonically balanced capitulation} (\S\
\ref{s:Capitulation}).
A common feature of all these \(p\)-class groups
\(G=\mathrm{Gal}(\mathrm{F}_p^1(K)/K)\simeq\mathrm{Cl}_p(K)\simeq C_{p^2}\times C_{p^2}\)
with \(p\in\lbrace 2,3\rbrace\)
is their consistence of \textit{four layers} of intermediate subgroups \(1\le S<G\)
corresponding to
four layers of unramified abelian \(p\)-extensions \(K<N\le\mathrm{F}_p^1(K)\)
within the Hilbert \(p\)-class field \(\mathrm{F}_p^1(K)\) of \(K\),
according to the Artin reciprocity law
\cite{Ar1927}
and the Galois correspondence.
For a coarse overview with cardinalities see Table
\ref{tbl:Layers},
for finer details about generators see \S\ 
\ref{s:NormalLattice},
in particular Figure
\ref{fig:NormalLattice}.


\renewcommand{\arraystretch}{1.1}

\begin{table}[ht]
\caption{Layers of subgroups and field extensions}
\label{tbl:Layers}
\begin{center}
\begin{tabular}{|c||c|c|c|}
\hline
 Layer              & Subgroup index resp. extension degree & \multicolumn{2}{c|}{Cardinality} \\
                    & \((G:S)=\lbrack N:K\rbrack=\)                 & \(p=2\) & \(p=3\)                \\
\hline
 \(\mathrm{Lyr}_1\) & \(p\)                                         & \(3\)   & \(4\)                  \\
 \(\mathrm{Lyr}_2\) & \(p^2\)                                       & \(7\)   & \(13\)                 \\
 \(\mathrm{Lyr}_3\) & \(p^3\)                                       & \(3\)   & \(4\)                  \\
 \(\mathrm{Lyr}_4\) & \(p^4\)                                       & \(1\)   & \(1\)                  \\
\hline
\end{tabular}
\end{center}
\end{table}


\noindent
Similarly as in
\cite{MaNm},
we shall use the strategy of \textit{pattern recognition} via Artin transfers:
we show that computational arithmetical information on the
transfer \(T_{K,N}:\,\mathrm{Cl}_3(K)\to\mathrm{Cl}_3(N)\) of \(3\)-classes
from the base field \(K\) to the fields \(N\) with relative degrees \(3\) and \(9\) in the first and second layer
suffices for determining the \textit{metabelianization} \(M=\mathfrak{G}/\mathfrak{G}^{\prime\prime}\)
of the \(3\)-\textit{class field tower group} \(\mathfrak{G}=\mathrm{Gal}(\mathrm{F}_3^\infty(K)/K)\) of \(K\),
i.e. the Galois group of the maximal unramified pro-\(3\) extension \(\mathrm{F}_3^\infty(K)\) of \(K\).
Data for fields with degrees \(27\) and \(81\) in the third and fourth layer
can be derived by purely group theoretic methods from a presentation of \(M\)
without the need of any number theoretic calculations.
In particular, the structure of the \(3\)-class group \(\mathrm{Cl}_3(\mathrm{F}_3^1(K))\)
of the first Hilbert \(3\)-class field \(\mathrm{F}_3^1(K)\) of \(K\),
which alone forms the fourth layer, can be obtained in this way.


\section{Arithmetic and algebraic Artin pattern}
\label{s:ArtinPattern}
\noindent
The input for the process of pattern recognition
consists of arithmetic information about an assigned number field,
collected in its Artin pattern.

\begin{definition}
\label{dfn:APinNT}
The \textit{number theoretic Artin pattern} \(\mathrm{AP}(K)\)
of an algebraic number field \(K\)
with respect to a prime \(p\)
is defined as the collection \((\varkappa(K),\tau(K))\)
of all kernels \(\varkappa(K)=(\ker(T_{K,N}))_N\) 
and targets \(\tau(K)=(\mathrm{Cl}_p(N))_N\) 
of extension homomorphisms
\(T_{K,N}:\,\mathrm{Cl}_p(K)\to\mathrm{Cl}_p(N)\)
of \(p\)-class groups
from \(K\) to unramified abelian extensions \(K\le N\le \mathrm{F}_p^1(K)\) of \(K\)
within the first Hilbert \(p\)-class field \(\mathrm{F}_p^1(K)\) of \(K\).
\end{definition}


\noindent
The idea of the number theoretic (arithmetic) Artin pattern
\(\mathrm{AP}(K)=(\varkappa(K),\tau(K))\)
was first introduced in
\cite{Ma2011},
using the terminology
transfer kernel type (TKT) \(\varkappa(K)\) and transfer target type (TTT) \(\tau(K)\).
It is determined by means of
class field theoretic routines of the computational algebra system Magma
\cite{BCP,BCFS,MAGMA}.

Before the arithmetic information
can be exploited as a search pattern for a database query in the SmallGroups Library
\cite{BEO}
it must be interpreted as algebraic information,
by means of the Artin reciprocity law of class field theory
\cite{Ar1927,Ar1929}.


\begin{definition}
\label{dfn:APinGT}
Given a prime \(p\),
the \textit{group theoretic Artin pattern} \(\mathrm{AP}(G)\)
of a pro-\(p\) group \(G\)
consists of all kernels \(\varkappa(G)=(\ker(T_{G,S}))_S\)
and targets \(\tau(G)=(S/S^\prime)_S\)
of Artin transfer homomorphisms
\(T_{G,S}:\,G/G^\prime\to S/S^\prime\)
from \(G\) to normal subgroups \(G^\prime\le S\le G\) of finite index \((G:S)<\infty\)
containing the commutator subgroup \(G^\prime\) of \(G\).
\end{definition}


\noindent
For the process of pattern recognition, the group \(G\) is taken as the
metabelian Galois group \(M=\mathrm{Gal}(\mathrm{F}_p^2(K)/K)\)
of the second Hilbert \(p\)-class field \(\mathrm{F}_p^2(K)\) of a number field \(K\),
briefly called the second \(p\)-class group of \(K\)
\cite{Ma2012}.
The following result admits the identification of
arithmetic and algebraic data.


\begin{theorem}
\label{thm:ArtinReciprocity}
Let \(p\) be a prime, then
the group theoretic Artin pattern \(\mathrm{AP}(M)\) of the second \(p\)-class group
\(M=\mathrm{Gal}(\mathrm{F}_p^2(K)/K)\) of \(K\)
coincides with the number theoretic Artin pattern \(\mathrm{AP}(K)\) of \(K\),
and the commutator subgroup \(M^\prime\) of \(M\)
is isomorphic to the \(p\)-class group
\(\mathrm{Cl}_p(\mathrm{F}_p^1(K))\simeq\mathrm{Gal}(\mathrm{F}_p^2(K)/\mathrm{F}_p^1(K))\)
of the first Hilbert \(p\)-class field \(\mathrm{F}_p^1(K)\) of \(K\).
\end{theorem}

\begin{proof}
See
\cite[\S\ 1.6, pp. 77--79]{Ma2016b},
in particular
\cite[Thm. 1.12, p. 77]{Ma2016b}.
\end{proof}


\section{Normal lattice and layers of \(G/G^\prime\simeq C_9\times C_9\)}
\label{s:NormalLattice}
\noindent
Let \(G\) be a finite \(3\)-group with two generators \(x,y\) such that
\(G=\langle x,y\rangle\) and \(G/G^\prime\simeq C_9\times C_9\),
that is, \(x^9,y^9\in G^\prime\) but \(x^3,y^3\notin G^\prime\).
Here \(G^\prime\) denotes the \textit{derived subgroup} of \(G\).
Let \(\mathrm{Lyr}_0(G)=\lbrace G\rbrace\) be the \textit{0th layer} with \(\tau_0=(G/G^\prime)\).
In the following description of the \textit{normal lattice} of \(G\) above \(G^\prime\),
we give generators for the members of each layer mentioned in Table
\ref{tbl:Layers}.


\begin{proposition}
\label{prp:Layers}
The \textbf{Frattini subgroup} \(J_0:=\Phi(G)=\langle x^3,y^3,G^\prime\rangle\) of \(G\)
is the intersection of the four \textbf{maximal subgroups} which form the \textbf{first layer} \(\mathrm{Lyr}_1(G)\):
\begin{equation}
\label{eqn:Lyr1}
H_1:=\langle x,J_0\rangle,\quad H_2:=\langle xy,J_0\rangle,\quad H_3:=\langle xy^2,J_0\rangle,\quad H_4:=\langle y,J_0\rangle,\quad
\end{equation}
with Artin pattern \((\varkappa_1,\tau_1)\).
They are of index \(3\) in \(G\) and can also be viewed as
\[
H_1=\langle x,y^3,G^\prime\rangle,\quad H_2=\langle xy,y^3,G^\prime\rangle,\quad H_3=\langle xy^2,y^3,G^\prime\rangle,\quad H_4=\langle y,x^3,G^\prime\rangle,
\]
which shows that \(H_i/G^\prime\simeq C_9\times C_3\), for each \(1\le i\le 4\),
whereas \(J_0/G^\prime\simeq C_3\times C_3\).

In addition to \(J_0\), there are twelve subgroups \(J_{ik}\)
with \(1\le i\le 4\), \(1\le k\le 3\), of index \(9\) in \(G\)
which belong to the \textbf{second layer} \(\mathrm{Lyr}_2(G)\):
\begin{equation}
\label{eqn:Lyr2}
\begin{aligned}
\text{three subgroups of } H_1,
\quad J_{11}&:=\langle x,G^\prime\rangle,\quad J_{12}:=\langle xy^3,G^\prime\rangle,\quad J_{13}:=\langle xy^6,G^\prime\rangle, \\
\text{three subgroups of } H_2,
\quad J_{21}&:=\langle xy,G^\prime\rangle,\quad J_{22}:=\langle xy^4,G^\prime\rangle,\quad J_{23}:=\langle xy^7,G^\prime\rangle, \\
\text{three subgroups of } H_3,
\quad J_{31}&:=\langle xy^2,G^\prime\rangle,\quad J_{32}:=\langle xy^5,G^\prime\rangle,\quad J_{33}:=\langle xy^8,G^\prime\rangle, \quad \text{ and} \\
\text{three subgroups of } H_4,
\quad J_{41}&:=\langle y,G^\prime\rangle,\quad J_{42}:=\langle x^3y,G^\prime\rangle,\quad J_{43}:=\langle x^6y,G^\prime\rangle,
\end{aligned}
\end{equation}
with Artin pattern \((\varkappa_2,\tau_2)\).
They all have a cyclic quotient \(J_{ik}/G^\prime\simeq C_9\).

Finally, there are four subgroups of index \(27\) in \(G\)
which form the \textbf{third layer} \(\mathrm{Lyr}_3(G)\):
\begin{equation}
\label{eqn:Lyr3}
K_1:=\langle x^3,G^\prime\rangle,\quad K_2:=\langle x^3y^3,G^\prime\rangle,\quad K_3:=\langle x^3y^6,G^\prime\rangle,\quad K_4:=\langle y^3,G^\prime\rangle,\quad
\end{equation}
with Artin pattern \((\varkappa_3,\tau_3)\).
They have quotient \(K_i/G^\prime\simeq C_3\), for each \(1\le i\le 4\).
\end{proposition}

\noindent
Additionally, let \(\mathrm{Lyr}_4(G)=\lbrace G^\prime\rbrace\) be the \textit{fourth layer} with \(\tau_4=(G^\prime/G^{\prime\prime})\).

Logarithmic abelian type invariants are used for the components of \(\tau\),
for instance,
\((311)\) or even \((31^2)\), with formal exponents indicating iteration, instead of \((27,3,3)\). See
\cite[\S\ 2.1, Eqn. (2.3), p. 661]{Ma2017b}.


\begin{figure}[ht]
\caption{Normal lattice of \(G/G^\prime\simeq C_9\times C_9\)}
\label{fig:NormalLattice}

{\tiny

\setlength{\unitlength}{1.0cm}
\begin{picture}(14,6)(-12,-10)



\put(-12.2,-4.9){\makebox(0,0)[rb]{(Part \(1\))}}
\put(-11,-5){\circle*{0.2}}
\put(-10.8,-4.9){\makebox(0,0)[lb]{\(G\)}}
\put(-10.8,-5.1){\makebox(0,0)[lt]{}}

\put(-11,-5){\line(-1,-1){2}}

\put(-12,-6){\circle{0.2}}
\put(-12.2,-5.9){\makebox(0,0)[rb]{\(H_1\)}}

\put(-12,-6){\line(-1,-2){0.4}}
\put(-12,-6){\line(1,-2){0.4}}
\put(-12,-6){\line(1,-1){1}}

\put(-13,-7){\circle{0.2}}
\put(-13.2,-7){\makebox(0,0)[rc]{\(J_{11}\)}}
\put(-12.4,-7){\circle{0.2}}
\put(-12.5,-7){\makebox(0,0)[rc]{\(J_{12}\)}}
\put(-11.6,-7){\circle{0.2}}
\put(-11.7,-7){\makebox(0,0)[rc]{\(J_{13}\)}}

\put(-10.8,-6.9){\makebox(0,0)[lb]{}}
\put(-11,-7){\circle*{0.2}}
\put(-10.8,-7.1){\makebox(0,0)[lt]{\(J_0\)}}

\put(-12,-8){\line(-1,2){0.4}}
\put(-12,-8){\line(1,2){0.4}}
\put(-12,-8){\line(1,1){1}}

\put(-12,-8){\circle{0.2}}
\put(-12.2,-8.1){\makebox(0,0)[rt]{\(K_1\)}}

\put(-11,-9){\line(-1,1){2}}

\put(-11,-9){\circle*{0.2}}
\put(-10.8,-8.9){\makebox(0,0)[lb]{}}
\put(-10.8,-9.1){\makebox(0,0)[lt]{\(G^\prime\)}}



\put(-8.2,-4.9){\makebox(0,0)[rb]{(Part \(2\))}}
\put(-7,-5){\circle*{0.2}}
\put(-6.8,-4.9){\makebox(0,0)[lb]{\(G\)}}
\put(-6.8,-5.1){\makebox(0,0)[lt]{}}

\put(-7,-5){\line(-1,-1){2}}

\put(-8,-6){\circle{0.2}}
\put(-8.2,-5.9){\makebox(0,0)[rb]{\(H_2\)}}

\put(-8,-6){\line(-1,-2){0.4}}
\put(-8,-6){\line(1,-2){0.4}}
\put(-8,-6){\line(1,-1){1}}

\put(-9,-7){\circle{0.2}}
\put(-9.2,-7){\makebox(0,0)[rc]{\(J_{21}\)}}
\put(-8.4,-7){\circle{0.2}}
\put(-8.5,-7){\makebox(0,0)[rc]{\(J_{22}\)}}
\put(-7.6,-7){\circle{0.2}}
\put(-7.7,-7){\makebox(0,0)[rc]{\(J_{23}\)}}

\put(-6.8,-6.9){\makebox(0,0)[lb]{}}
\put(-7,-7){\circle*{0.2}}
\put(-6.8,-7.1){\makebox(0,0)[lt]{\(J_0\)}}

\put(-8,-8){\line(-1,2){0.4}}
\put(-8,-8){\line(1,2){0.4}}
\put(-8,-8){\line(1,1){1}}

\put(-8,-8){\circle{0.2}}
\put(-8.2,-8.1){\makebox(0,0)[rt]{\(K_2\)}}

\put(-7,-9){\line(-1,1){2}}

\put(-7,-9){\circle*{0.2}}
\put(-6.8,-8.9){\makebox(0,0)[lb]{}}
\put(-6.8,-9.1){\makebox(0,0)[lt]{\(G^\prime\)}}



\put(-4.2,-4.9){\makebox(0,0)[rb]{(Part \(3\))}}
\put(-3,-5){\circle*{0.2}}
\put(-2.8,-4.9){\makebox(0,0)[lb]{\(G\)}}
\put(-2.8,-5.1){\makebox(0,0)[lt]{}}

\put(-3,-5){\line(-1,-1){2}}

\put(-4,-6){\circle{0.2}}
\put(-4.2,-5.9){\makebox(0,0)[rb]{\(H_3\)}}

\put(-4,-6){\line(-1,-2){0.4}}
\put(-4,-6){\line(1,-2){0.4}}
\put(-4,-6){\line(1,-1){1}}

\put(-5,-7){\circle{0.2}}
\put(-5.2,-7){\makebox(0,0)[rc]{\(J_{31}\)}}
\put(-4.4,-7){\circle{0.2}}
\put(-4.5,-7){\makebox(0,0)[rc]{\(J_{32}\)}}
\put(-3.6,-7){\circle{0.2}}
\put(-3.7,-7){\makebox(0,0)[rc]{\(J_{33}\)}}

\put(-2.8,-6.9){\makebox(0,0)[lb]{}}
\put(-3,-7){\circle*{0.2}}
\put(-2.8,-7.1){\makebox(0,0)[lt]{\(J_0\)}}

\put(-4,-8){\line(-1,2){0.4}}
\put(-4,-8){\line(1,2){0.4}}
\put(-4,-8){\line(1,1){1}}

\put(-4,-8){\circle{0.2}}
\put(-4.2,-8.1){\makebox(0,0)[rt]{\(K_3\)}}

\put(-3,-9){\line(-1,1){2}}

\put(-3,-9){\circle*{0.2}}
\put(-2.8,-8.9){\makebox(0,0)[lb]{}}
\put(-2.8,-9.1){\makebox(0,0)[lt]{\(G^\prime\)}}



\put(-0.2,-4.9){\makebox(0,0)[rb]{(Part \(4\))}}
\put(1,-5){\circle*{0.2}}
\put(1.2,-4.9){\makebox(0,0)[lb]{\(G\)}}
\put(1.2,-5.1){\makebox(0,0)[lt]{full group}}

\put(1,-5){\line(-1,-1){2}}

\put(0,-6){\circle{0.2}}
\put(-0.2,-5.9){\makebox(0,0)[rb]{\(H_4\)}}

\put(0,-6){\line(-1,-2){0.4}}
\put(0,-6){\line(1,-2){0.4}}
\put(0,-6){\line(1,-1){1}}

\put(-1,-7){\circle{0.2}}
\put(-1.2,-7){\makebox(0,0)[rc]{\(J_{41}\)}}
\put(-0.4,-7){\circle{0.2}}
\put(-0.5,-7){\makebox(0,0)[rc]{\(J_{42}\)}}
\put(0.4,-7){\circle{0.2}}
\put(0.3,-7){\makebox(0,0)[rc]{\(J_{43}\)}}

\put(1.2,-6.9){\makebox(0,0)[lb]{Frattini subg.}}
\put(1,-7){\circle*{0.2}}
\put(1.2,-7.1){\makebox(0,0)[lt]{\(J_0=\Phi(G)\)}}

\put(0,-8){\line(-1,2){0.4}}
\put(0,-8){\line(1,2){0.4}}
\put(0,-8){\line(1,1){1}}

\put(0,-8){\circle{0.2}}
\put(-0.2,-8.1){\makebox(0,0)[rt]{\(K_4\)}}

\put(1,-9){\line(-1,1){2}}

\put(1,-9){\circle*{0.2}}
\put(1.2,-8.9){\makebox(0,0)[lb]{derived subg.}}
\put(1.2,-9.1){\makebox(0,0)[lt]{\(G^\prime\)}}


\end{picture}

}

\end{figure}
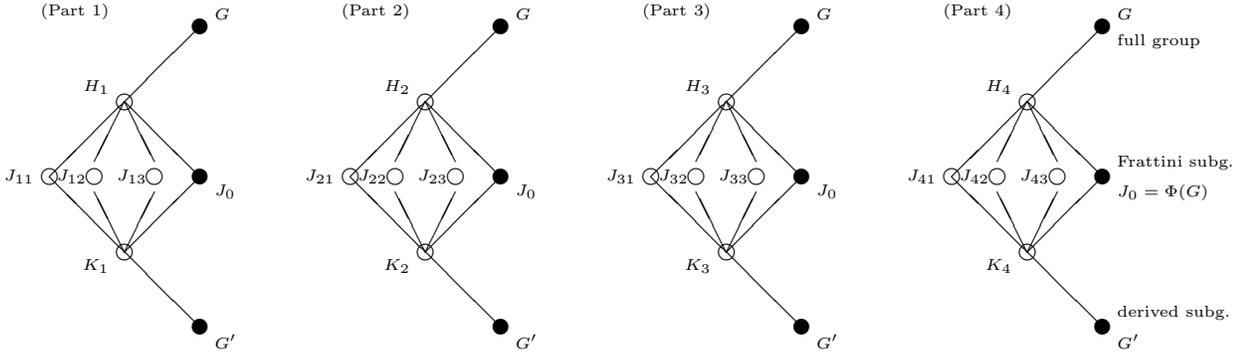

Figure
\ref{fig:NormalLattice}
shows that the normal lattice of \(G\) above the derived subgroup \(G^\prime\)
consists of four similar parts (\(1\le i\le 4\)),
agglutinated at the subgroups \(G\), \(J_0=\Phi(G)\) and \(G^\prime\).


\section{Cohomology criterion and Schur \(\sigma\)-groups}
\label{s:Shafarevich}
\noindent
Let \(p\) be a prime number and \(G\) be a pro-\(p\) group.
Then the finite field \(\mathbb{F}_p\) is a trivial \(G\)-module,
and two crucial cohomological invariants of \(G\) are
the \textit{generator rank} \(d_1:=d_{p,1}(G):=\dim_{\mathbb{F}_p}\mathrm{H}^1(G,\mathbb{F}_p)\) and
the \textit{relation rank} \(d_2:=d_{p,2}(G):=\dim_{\mathbb{F}_p}\mathrm{H}^2(G,\mathbb{F}_p)\).
Shafarevich has established the following necessary criterion.


\begin{theorem}
\label{thm:Shafarevich}
The Galois group
\(G=\mathrm{Gal}(\mathrm{F}_p^\infty(K)/K)\)
of the Hilbert \(p\)-class field tower \(\mathrm{F}_p^\infty(K)\) of a number field \(K\),
that is the maximal unramified pro-\(p\) extension of \(K\),
is confined to constraints in terms of the generator rank and relation rank of \(G\):
\[
d_1\le d_2\le d_1+r+\theta,
\]
where \(r\) and \(\theta\) are defined in the following way:
if \((r_1,r_2)\) denotes the signature of \(K\),
then the torsion free \textit{Dirichlet unit rank} of \(K\) is given by \(r=r_1+r_2-1\)
and
\[
\theta=
\begin{cases}
1 & \text{if } K \text{ contains a primitive } p\text{-th root of unity}, \\
0 & \text{otherwise}.
\end{cases}
\]
\end{theorem}

\begin{proof}
See
\cite[Thm. 5.1, p. 28]{Ma2015}.
\end{proof}


\begin{corollary}
\label{cor:Shafarevich}
Applied to the particular case of \(p=3\)
and the \(3\)-class field tower group \(\mathfrak{G}=\mathrm{Gal}(\mathrm{F}_3^\infty(K)/K)\)
of a quadratic number field \(K\) with \(3\)-class group
\(\mathfrak{G}/\mathfrak{G}^\prime\simeq\mathrm{Cl}_3(K)\simeq C_9\times C_9\),
and thus with Frattini quotient \(\mathfrak{G}/\Phi(\mathfrak{G})\simeq C_3\times C_3\),
the \textit{Shafarevich criterion} becomes
\[
2\le d_2\le 2+r,
\]
where
\[
r=
\begin{cases}
1 & \text{for real } K \text{ with } (r_1,r_2)=(2,0), \\
0 & \text{for imaginary } K \text{ with } (r_1,r_2)=(0,1).
\end{cases}
\]
\end{corollary}

\begin{proof}
The generator rank of \(\mathfrak{G}\) is \(d_1=2\), according to the Burnside basis theorem,
\(\theta=0\) because the cyclotomic quadratic field \(\mathbb{Q}(\sqrt{-3})\) has class number \(h=1\),
and the Dirichlet unit rank \(r\) of \(K\) is given in the indicated way.
\end{proof}


For \textit{imaginary} quadratic base fields, the cohomology criterion is sharp, \(d_2=d_1=2\),
and motivates the following definitions.

\begin{definition}
\label{dfn:Schur}
The pro-\(p\) group \(G\) is called a \textit{Schur group}
if it has a balanced presentation with coinciding relation rank and generator rank, \(d_2=d_1\).

If there exists an automorphism \(\sigma\in\mathrm{Aut}(G)\)
which acts as inversion on \(\mathrm{H}^1(G,\mathbb{F}_p)\), that is
\((\forall\,x\in\mathrm{H}^1(G,\mathbb{F}_p))\,x^\sigma=x^{-1}\),
then \(G\) is called a \(\sigma1\)-\textit{group} or briefly \(\sigma\)-\textit{group}.

\noindent
If \(\sigma\) additionally acts as inversion on \(\mathrm{H}^2(G,\mathbb{F}_p)\),
then \(G\) is called a \(\sigma2\)-\textit{group}.
\end{definition}


\begin{theorem}
\label{thm:Schur}
For an odd prime \(p\ge 3\),
the \textbf{\(p\)-class field tower group} \(\mathfrak{G}=\mathrm{Gal}(\mathrm{F}_p^\infty(K)/K)\)
of an imaginary quadratic field must be a Schur \(\sigma\)-group.
More generally,
the \(p\)-class field tower group \(\mathfrak{G}\)
and all \textbf{higher \(p\)-class groups} \(\mathrm{Gal}(\mathrm{F}_p^n(K)/K)\), \(n\ge 2\),
must be \(\sigma2\)-groups.
\end{theorem}

\begin{proof}
The first result is due to Koch and Venkov
\cite{KoVe}.
See also the papers by Arrigoni
\cite{Ag}
and Boston, Bush, Hajir
\cite{BBH1},
where a \(\sigma\)-group is called a
group with generator inverting (GI) automorphism.
The second claim was proved by Schoof
\cite{Sf}.
See also Boston, Bush, Hajir
\cite{BBH2},
and our paper
\cite[Dfn. 4.3, p. 83]{Ma2018},
where a \(\sigma2\)-group is called a
group with relator inverting (RI) automorphism.
\end{proof}


\renewcommand{\arraystretch}{1.1}

\begin{table}[ht]
\caption{Step size-\(3\) children of \(R=\langle 243,2\rangle\)}
\label{tbl:Ord243Id2}
\begin{center}
\begin{tabular}{|c|c||c|c|l|}
\hline
 \(i\)  & SmallGroup                 & \(\nu\) & \(\mu\) & \((N_i,C_i)_{1\le i\le\nu}\) \\
\hline
 \(22\) & \(\langle 6561,23\rangle\) & \(3\)   & \(5\)   & \((10,0;22,7;9,9)\)          \\
 \(23\) & \(\langle 6561,24\rangle\) & \(2\)   & \(4\)   & \((2,0;3,0)\)                \\
 \(24\) & \(\langle 6561,25\rangle\) & \(3\)   & \(5\)   & \((7,0;14,5;6,6)\)           \\
 \(25\) & \(\langle 6561,26\rangle\) & \(2\)   & \(4\)   & \((4,0;3,0)\)                \\
 \(26\) & \(\langle 6561,27\rangle\) & \(2\)   & \(4\)   & \((4,0;3,0)\)                \\
 \(27\) & \(\langle 6561,28\rangle\) & \(2\)   & \(4\)   & \((3,0;5,0)\)                \\
\hline
\end{tabular}
\end{center}
\end{table}


\section{Metabelian root of type \((9,9)\)}
\label{s:Metabelian}
\noindent
A crucial non-abelian root of finite \(3\)-groups \(G\) with derived quotient \(G/G^\prime\simeq C_9\times C_9\)
is the metabelian \(3\)-group \(R:=\langle 243,2\rangle\) of order \(243=3^5\).

\begin{proposition}
\label{prp:MetabelianRoot}
\(R\) is a highly capable \(\sigma\)-group with
nuclear rank \(\nu(R)=5\), \(p\)-multiplicator rank \(\mu(R)=5\) and descendant numbers
\((N_i,C_i)_{1\le i\le 5}=(11,2;58,32;58,58;11,11;1,1)\).
Invariants and SmallGroup identifiers
\cite{MAGMA6561}
for several crucial immediate descendants of step size \(s=3\) of \(R\)
are given in Table
\ref{tbl:Ord243Id2}.
\end{proposition}


\noindent
With respect to \textit{imaginary} quadratic fields of type \((9,9)\),
the immediate descendants \(R-\#3;i\) with step size \(s=3\) of \(R\),
given in Proposition
\ref{prp:MetabelianRoot},
are most important.
Their locations in the descendant tree of \(R\) are drawn in Figure
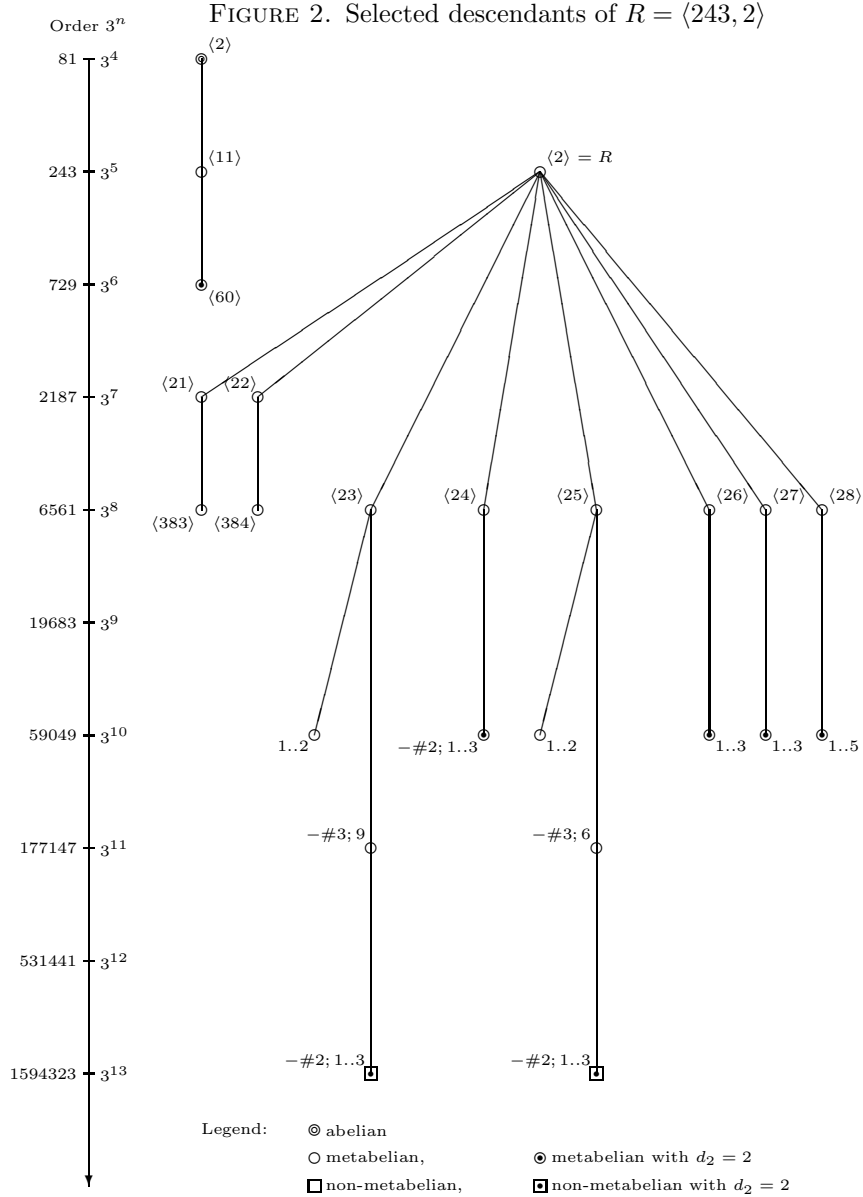
\ref{fig:RootRegion}.


\begin{figure}[ht]
\caption{Selected descendants of \(R=\langle 243,2\rangle\)}
\label{fig:RootRegion}

{\tiny

\setlength{\unitlength}{0.75cm}
\begin{picture}(10,21)(-6,-20.5)

\put(-8,0.5){\makebox(0,0)[cb]{Order \(3^n\)}}
\put(-8,0){\line(0,-1){18}}
\multiput(-8.1,0)(0,-2){10}{\line(1,0){0.2}}
\put(-8.2,0){\makebox(0,0)[rc]{\(81\)}}
\put(-7.8,0){\makebox(0,0)[lc]{\(3^4\)}}
\put(-8.2,-2){\makebox(0,0)[rc]{\(243\)}}
\put(-7.8,-2){\makebox(0,0)[lc]{\(3^5\)}}
\put(-8.2,-4){\makebox(0,0)[rc]{\(729\)}}
\put(-7.8,-4){\makebox(0,0)[lc]{\(3^6\)}}
\put(-8.2,-6){\makebox(0,0)[rc]{\(2187\)}}
\put(-7.8,-6){\makebox(0,0)[lc]{\(3^7\)}}
\put(-8.2,-8){\makebox(0,0)[rc]{\(6561\)}}
\put(-7.8,-8){\makebox(0,0)[lc]{\(3^8\)}}
\put(-8.2,-10){\makebox(0,0)[rc]{\(19683\)}}
\put(-7.8,-10){\makebox(0,0)[lc]{\(3^9\)}}
\put(-8.2,-12){\makebox(0,0)[rc]{\(59049\)}}
\put(-7.8,-12){\makebox(0,0)[lc]{\(3^{10}\)}}
\put(-8.2,-14){\makebox(0,0)[rc]{\(177147\)}}
\put(-7.8,-14){\makebox(0,0)[lc]{\(3^{11}\)}}
\put(-8.2,-16){\makebox(0,0)[rc]{\(531441\)}}
\put(-7.8,-16){\makebox(0,0)[lc]{\(3^{12}\)}}
\put(-8.2,-18){\makebox(0,0)[rc]{\(1594323\)}}
\put(-7.8,-18){\makebox(0,0)[lc]{\(3^{13}\)}}
\put(-8,-18){\vector(0,-1){2}}


\multiput(-6,0)(0,-2){3}{\circle{0.2}}
\put(-6,0){\circle{0.1}}
\put(-6,0){\line(0,-1){2}}
\put(-6,-2){\line(0,-1){2}}
\put(-6,-4){\circle*{0.1}}

\put(0,-2){\circle{0.2}}

\put(0,-2){\line(-3,-2){6}}
\put(0,-2){\line(-5,-4){5}}
\put(-6,-6){\circle{0.2}}
\put(-5,-6){\circle{0.2}}
\put(-6,-6){\line(0,-1){2}}
\put(-5,-6){\line(0,-1){2}}
\put(-6,-8){\circle{0.2}}
\put(-5,-8){\circle{0.2}}

\put(0,-2){\line(-1,-2){3}}
\put(0,-2){\line(-1,-6){1}}
\put(0,-2){\line(1,-6){1}}
\put(0,-2){\line(1,-2){3}}
\put(0,-2){\line(2,-3){4}}
\put(0,-2){\line(5,-6){5}}
\multiput(-3,-8)(2,0){3}{\circle{0.2}}
\multiput(3,-8)(1,0){3}{\circle{0.2}}

\put(-3,-8){\line(-1,-4){1}}
\put(-3,-8){\line(0,-1){6}}
\put(-1,-8){\line(0,-1){4}}
\put(1,-8){\line(-1,-4){1}}
\put(1,-8){\line(0,-1){6}}
\multiput(3,-8)(1,0){3}{\line(0,-1){4}}

\put(-4,-12){\circle{0.2}}
\put(-1,-12){\circle{0.2}}
\put(-1,-12){\circle*{0.1}}
\put(0,-12){\circle{0.2}}
\put(3,-12){\circle{0.2}}
\put(3,-12){\circle*{0.1}}
\put(4,-12){\circle{0.2}}
\put(4,-12){\circle*{0.1}}
\put(5,-12){\circle{0.2}}
\put(5,-12){\circle*{0.1}}

\put(-3,-14){\circle{0.2}}
\put(1,-14){\circle{0.2}}
\put(-3,-14){\line(0,-1){4}}
\put(1,-14){\line(0,-1){4}}
\put(-3.1,-18.1){\framebox(0.2,0.2){}}
\put(-3,-18){\circle*{0.1}}
\put(0.9,-18.1){\framebox(0.2,0.2){}}
\put(1,-18){\circle*{0.1}}


\put(-5.9,0.1){\makebox(0,0)[lb]{\(\langle 2\rangle\)}}
\put(-5.9,-1.9){\makebox(0,0)[lb]{\(\langle 11\rangle\)}}
\put(-5.9,-4.1){\makebox(0,0)[lt]{\(\langle 60\rangle\)}}

\put(0.1,-1.9){\makebox(0,0)[lb]{\(\langle 2\rangle=R\)}}

\put(-6.1,-5.9){\makebox(0,0)[rb]{\(\langle 21\rangle\)}}
\put(-5.0,-5.9){\makebox(0,0)[rb]{\(\langle 22\rangle\)}}
\put(-6.1,-8.1){\makebox(0,0)[rt]{\(\langle 383\rangle\)}}
\put(-5.0,-8.1){\makebox(0,0)[rt]{\(\langle 384\rangle\)}}

\put(-3.1,-7.9){\makebox(0,0)[rb]{\(\langle 23\rangle\)}}
\put(-1.1,-7.9){\makebox(0,0)[rb]{\(\langle 24\rangle\)}}
\put(0.9,-7.9){\makebox(0,0)[rb]{\(\langle 25\rangle\)}}
\put(3.1,-7.9){\makebox(0,0)[lb]{\(\langle 26\rangle\)}}
\put(4.1,-7.9){\makebox(0,0)[lb]{\(\langle 27\rangle\)}}
\put(5.1,-7.9){\makebox(0,0)[lb]{\(\langle 28\rangle\)}}

\put(-4.1,-12.1){\makebox(0,0)[rt]{\(1..2\)}}
\put(-1.1,-12.1){\makebox(0,0)[rt]{\(-\#2;1..3\)}}
\put(0.1,-12.1){\makebox(0,0)[lt]{\(1..2\)}}
\put(3.1,-12.1){\makebox(0,0)[lt]{\(1..3\)}}
\put(4.1,-12.1){\makebox(0,0)[lt]{\(1..3\)}}
\put(5.1,-12.1){\makebox(0,0)[lt]{\(1..5\)}}

\put(-3.1,-13.9){\makebox(0,0)[rb]{\(-\#3;9\)}}
\put(0.9,-13.9){\makebox(0,0)[rb]{\(-\#3;6\)}}
\put(-3.1,-17.9){\makebox(0,0)[rb]{\(-\#2;1..3\)}}
\put(0.9,-17.9){\makebox(0,0)[rb]{\(-\#2;1..3\)}}


\put(-6,-19){\makebox(0,0)[lc]{Legend:}}
\put(-4,-19){\circle{0.2}}
\put(-4,-19){\circle{0.1}}
\put(-3.8,-19){\makebox(0,0)[lc]{abelian}}
\put(-4,-19.5){\circle{0.2}}
\put(-3.8,-19.5){\makebox(0,0)[lc]{metabelian,}}
\put(0,-19.5){\circle{0.2}}
\put(0,-19.5){\circle*{0.1}}
\put(0.2,-19.5){\makebox(0,0)[lc]{metabelian with \(d_2=2\)}}
\put(-4.1,-20.1){\framebox(0.2,0.2){}}
\put(-3.8,-20){\makebox(0,0)[lc]{non-metabelian,}}
\put(-0.1,-20.1){\framebox(0.2,0.2){}}
\put(0,-20){\circle*{0.1}}
\put(0.2,-20){\makebox(0,0)[lc]{non-metabelian with \(d_2=2\)}}

\end{picture}
}
\end{figure}


\section{Periodic trifurcations}
\label{s:Multifurcation}
\noindent

\begin{theorem}
\label{thm:Trifurcation}
(Tree with periodic trifurcations and parametrized sequences of Schur \(\sigma\)-groups) \\
For any upper bound \(u\le 30\),
there exists a sequence \((V_i)_{0\le i\le u}\) of \(\sigma\)-groups
with nuclear rank \(\nu(V_i)=3\) such that \((\forall\,0\le i<u)\)
\begin{enumerate}
\item
\(V_i=\pi_3(V_{i+1})\) is parent of \(V_{i+1}\) with step size \(s=3\),
\(V_{i+1}\) is the \textbf{unique} \(\sigma\)-child of \(V_i\) with step size \(s=3\)
and nuclear rank \(\nu(V_{i+1})=3\),
and
\item
\(V_i\) possesses a \textbf{unique} \(\sigma\)-child \(S_{i+1}\) with step size \(s=3\), \(V_i=\pi_3(S_{i+1})\),
with nuclear rank \(\nu(S_{i+1})=2\) and descendant numbers \((N_2,C_2)=(3,0)\).
\(S_{i+1}\) is a sibling of \(V_{i+1}\).
\end{enumerate}
All three immediate descendants \(T_{i+1,j}\), \(1\le j\le 3\), of \(S_{i+1}\) with step size \(s=2\), \(S_{i+1}=\pi_2(T_{i+1,j})\),
are non-metabelian Schur \(\sigma\)-groups \(T_{i+1,j}=S_{i+1}-\#2;j\), \(1\le j\le 3\).

The minimal possible selection of the root is
either \(V_0=\langle 243,2\rangle-\#3;22=\langle 6561,23\rangle\)
or \(V_0=\langle 243,2\rangle-\#3;24=\langle 6561,25\rangle\).

For this minimal selection,
the metabelianization of the non-metabelian Schur \(\sigma\)-groups \(T_{i+1,j}=S_{i+1}-\#2;j\), \(1\le j\le 3\),
is given by
\begin{equation}
\label{eqn:Metabelianization}
T_{i+1,j}/T_{i+1,j}^{\prime\prime}\simeq
\begin{cases}
S_{i+1} & \text{ if } i=0, \\
M & \text{ if } i\ge 1,
\end{cases}
\end{equation}
for a fixed step size-\(1\) child \(M\) of \(V_1\).

The metabelianization \(M\) is given by
\(V_1-\#1;4\) for the selection \(V_0=\langle 6561,23\rangle\),
resp. \(V_1-\#1;3\) for the selection \(V_0=\langle 6561,25\rangle\).
See Figure
\ref{fig:Trifurcations}.
\end{theorem}

\begin{proof}
All claims have been verified with the aid of Magma
\cite{MAGMA}
up to \(u=30\) and logarithmic order \(100\),
that is, for all descendants of orders not exceeding \(3^{100}\).
\end{proof}


\begin{corollary}
\label{cor:Trifurcation}
\(S_1\) is the metabelianization with \textbf{child topology}
of \(T_{1,j}\) with \(1\le j\le 3\), and
\(M\) is the common metabelianization with \textbf{fork topology}
of all \(T_{i,j}\) with \(i\ge 2\) and \(1\le j\le 3\).
The common \textbf{fork} for all \(i\ge 2\) is \(V_1\).
(These concepts have been introduced in
\cite[\S\ 5, pp. 89--92]{Ma2016b}.)
\end{corollary}


\begin{conjecture}
\label{cnj:Trifurcation}
Theorem
\ref{thm:Trifurcation}
remains true for any \(u>30\).
\end{conjecture}


\begin{figure}[ht]
\caption{Periodic trifurcations and Schur \(\sigma\)-groups}
\label{fig:Trifurcations}

{\tiny

\setlength{\unitlength}{0.75cm}
\begin{picture}(10,11)(-6,-10.5)

\put(-8,0.5){\makebox(0,0)[cb]{Order \(3^n\)}}
\put(-8,0){\line(0,-1){9}}
\multiput(-8.1,0)(0,-1){10}{\line(1,0){0.2}}
\put(-8.2,0){\makebox(0,0)[rc]{\(6561\)}}
\put(-7.8,0){\makebox(0,0)[lc]{\(3^8\)}}
\put(-8.2,-1){\makebox(0,0)[rc]{\(19683\)}}
\put(-7.8,-1){\makebox(0,0)[lc]{\(3^9\)}}
\put(-8.2,-2){\makebox(0,0)[rc]{\(59049\)}}
\put(-7.8,-2){\makebox(0,0)[lc]{\(3^{10}\)}}
\put(-8.2,-3){\makebox(0,0)[rc]{\(177147\)}}
\put(-7.8,-3){\makebox(0,0)[lc]{\(3^{11}\)}}
\put(-8.2,-4){\makebox(0,0)[rc]{\(531441\)}}
\put(-7.8,-4){\makebox(0,0)[lc]{\(3^{12}\)}}
\put(-8.2,-5){\makebox(0,0)[rc]{\(1594323\)}}
\put(-7.8,-5){\makebox(0,0)[lc]{\(3^{13}\)}}
\put(-8.2,-6){\makebox(0,0)[rc]{\(4782969\)}}
\put(-7.8,-6){\makebox(0,0)[lc]{\(3^{14}\)}}
\put(-8.2,-7){\makebox(0,0)[rc]{\(14348907\)}}
\put(-7.8,-7){\makebox(0,0)[lc]{\(3^{15}\)}}
\put(-8.2,-8){\makebox(0,0)[rc]{\(43046721\)}}
\put(-7.8,-8){\makebox(0,0)[lc]{\(3^{16}\)}}
\put(-8.2,-9){\makebox(0,0)[rc]{\(129140163\)}}
\put(-7.8,-9){\makebox(0,0)[lc]{\(3^{17}\)}}
\put(-8,-9){\vector(0,-1){1}}

\put(0,-9){\vector(2,-3){0.5}}
\put(0.6,-9.8){\makebox(0,0)[lc]{infinite main trunk}}


\put(-6,0){\circle{0.2}}
\put(-6,0){\line(0,-1){1}}
\put(-6,-1){\circle{0.2}}
\put(-6,0){\line(1,-2){1}}
\put(-5,-2){\circle{0.2}}
\put(-6,0){\line(2,-3){2}}
\put(-4,-3){\circle{0.2}}
\put(-6,0){\line(2,-1){6}}
\put(0,-3){\circle{0.2}}
\put(0,-3){\line(2,-1){4}}
\put(3.9,-5.1){\framebox(0.2,0.2){}}
\put(4,-5){\circle*{0.1}}


\put(-4,-3){\line(0,-1){1}}
\put(-4.1,-4.1){\framebox(0.2,0.2){}}
\put(-4,-3){\line(1,-1){1}}
\put(-3,-4){\circle{0.2}}
\put(-4,-3){\line(1,-2){1}}
\put(-3.1,-5.1){\framebox(0.2,0.2){}}
\put(-4,-3){\line(2,-3){2}}
\put(-2.1,-6.1){\framebox(0.2,0.2){}}
\put(-4,-3){\line(2,-1){6}}
\put(1.9,-6.1){\framebox(0.2,0.2){}}
\put(2,-6){\line(2,-1){4}}
\put(5.9,-8.1){\framebox(0.2,0.2){}}
\put(6,-8){\circle*{0.1}}


\put(-2,-6){\line(0,-1){1}}
\put(-2.1,-7.1){\framebox(0.2,0.2){}}
\put(-2,-6){\line(1,-2){1}}
\put(-1.1,-8.1){\framebox(0.2,0.2){}}
\put(-2,-6){\line(2,-3){2}}
\put(-0.1,-9.1){\framebox(0.2,0.2){}}
\put(-2,-6){\line(2,-1){6}}
\put(3.9,-9.1){\framebox(0.2,0.2){}}


\put(-6.3,0.1){\makebox(0,0)[rb]{\(V_0\)}}
\put(-4.3,-2.9){\makebox(0,0)[rb]{\(V_1\)}}
\put(-2.3,-5.9){\makebox(0,0)[rb]{\(V_2\)}}
\put(-0.3,-8.9){\makebox(0,0)[rb]{\(V_3\)}}

\put(-2.8,-4.1){\makebox(0,0)[lt]{\(M\)}}

\put(0.2,-2.9){\makebox(0,0)[lb]{\(S_1\)}}
\put(2.2,-5.9){\makebox(0,0)[lb]{\(S_2\)}}
\put(4.2,-8.9){\makebox(0,0)[lb]{\(S_3\)}}

\put(4.2,-5.1){\makebox(0,0)[lt]{\(T_{1,1..3}\)}}
\put(6.2,-8.1){\makebox(0,0)[lt]{\(T_{2,1..3}\)}}


\put(-2,0){\makebox(0,0)[lc]{Legend:}}

\put(-2,-0.5){\circle{0.2}}
\put(-1.8,-0.5){\makebox(0,0)[lc]{metabelian,}}

\put(-2.1,-1.1){\framebox(0.2,0.2){}}
\put(-1.8,-1){\makebox(0,0)[lc]{non-metabelian,}}
\put(1.9,-1.1){\framebox(0.2,0.2){}}
\put(2,-1){\circle*{0.1}}
\put(2.2,-1){\makebox(0,0)[lc]{non-metabelian with \(d_2=2\)}}

\end{picture}
}
\end{figure}

\noindent
In these descendant trees,
derived length \(\mathrm{dl}=3\) sets in partially with order \(3^{12}\), completely with order \(3^{13}\),
\(\mathrm{dl}=4\) partially with order \(3^{27}\), completely with order \(3^{29}\),
and \(\mathrm{dl}=5\) partially with order \(3^{60}\), completely with order \(3^{61}\).


\section{Capitulation laws}
\label{s:Capitulation}
\noindent
With \(p\) a prime number,
let \(G\) be a finite \(p\)-group.


\begin{definition}
\label{dfn:Harmonic}
Let \(\lbrace S_1,\ldots,S_n\rbrace\) be the set of intermediate subgroups \(G^\prime\le S\le G\)
between \(G\) and its derived subgroup \(G^\prime\),
let \(T_{G,S_i}:\,G/G^\prime\to S_i/S_i^\prime\) be
the collection of Artin transfers from \(G\) to \(S_i\), \(1\le i\le n\),
and denote by \(\mathfrak{S}_n\) the symmetric group on \(n\) letters.
The group \(G\) is said to possess \textit{harmonically balanced capitulation}
if there exists some permutation \(\pi\in\mathfrak{S}_n\) such that
\begin{equation}
\label{eqn:Harmonic}
(\forall 1\le i\le n)\quad\ker(T_{G,S_i})=S_{\pi(i)}/G^\prime, \text{ abbreviated by } \ker(T_{G,S_i})=S_{\pi(i)}.
\end{equation}
\end{definition}


\begin{remark}
\label{rmk:Arithmetic}
We shall restrict our attention to groups \(G\)
with \textit{arithmetical capitulation}
satisfying the following two axioms
\((\forall S\le G)\) (\(\ker(T_{G,S})=G^\prime\) \(\Longrightarrow\) \(S=G\)) and \(\ker(T_{G,G^\prime})=G\),
motivated by Hilbert's Theorem 94
and the Artin-Furtw\"angler principal ideal theorem.
\end{remark}


\begin{example}
\label{rmk:Elementary}
In
\cite{Ma2012}
and
\cite{Ma2011},
we thoroughly investigated the capitulation over quadratic fields \(K\)
with elementary bicyclic \(p\)-class group \(\mathrm{Cl}_p(K)\simeq C_p\times C_p\) for \(p=3\).
In
\cite{Ma2016b},
we looked at the analogue for \(p\in\lbrace 5,7\rbrace\).
Since arithmetical capitulation requires the mandatory transposition \(G\leftrightarrow G^\prime\),
harmonically balanced capitulation can be described by a permutation \(\pi\in\mathfrak{S}_{p+1}\)
of the \(p+1\) proper intermediate subgroups \(G^\prime<S_1,\ldots,S_{p+1}<G\).

For \(p=3\),
only two permutations \(\pi\in\mathfrak{S}_4\) are admissible.
They were called capitulation type \\
\(\mathrm{G}.16\), \(\varkappa=(1243)\) (two fixed points and a transposition), and \\
\(\mathrm{G}.19\), \(\varkappa=(2143)\) (two disjoint transpositions),
\cite[Tbl. 3--4, pp. 497--498]{Ma2012}, \\
and they enforce a \(3\)-class field tower with at least three stages, \(\ell_3(K)\ge 3\)
\cite[Tbl. 1, p. 90]{Ma2016b}.

For \(p=5\), however,
a lot of permutations \(\pi\in\mathfrak{S}_6\) can occur
\cite[Tbl. 3, p. 94]{Ma2016b},
for instance five cases which admit
a metabelian \(5\)-class field tower with two stages
\cite[Thm. 7.4, pp. 96-97]{Ma2016b},
the identity with six fixed points,
a \(4\)-cycle with two fixed points,
a \(5\)-cycle with a single fixed point,
two disjoint \(3\)-cycles, and
a \(6\)-cycles.
The latter two permutations are fixed point free,
just as type \(\mathrm{G}.19\) for \(p=3\),
which is described by Scholz with the following allegory:
\lq\lq So wie es sich f\"ur wohlerzogene B\"urger geziemt:
jeder l\"a\ss t einem anderen den Vortritt
aber jeder kommt auch selbst heran\rq\rq
\cite[pp. 155--156]{LmRq}.

Similarly, for \(p=7\)
there are also many admissible permutations \(\pi\in\mathfrak{S}_8\)
\cite[Tbl. 4, p. 95]{Ma2016b}.
\end{example}


\noindent
Not every permutation \(\pi\in\mathfrak{S}_n\) will be admissible in Formula
\eqref{eqn:Harmonic},
because there exist some general laws which have to be obeyed by the capitulation.
We state these laws for the special case of the normal lattice in \S\
\ref{s:NormalLattice}.
So let \(G\) be a finite \(3\)-group with
non-elementary bihomocyclic \(G/G^\prime\simeq C_9\times C_9\).


\begin{proposition}
\label{prp:Inclusion}
The capitulation kernels satisfy the following inclusions:
\begin{enumerate}
\item
\((\forall 1\le i\le 4)\,(\forall 1\le k\le 3)\quad G^\prime<\ker(T_{G,H_i})\le\ker(T_{G,J_{ik}})\le\ker(T_{G,K_i})\le G\),
\item
\((\forall 1\le i,j\le 4)\quad G^\prime<\ker(T_{G,H_i})\le\ker(T_{G,J_0})\le\ker(T_{G,K_j})\le G\).
\end{enumerate}
\end{proposition}

\begin{proof}
For nested subgroups \(K<H<G\), we generally have the compositum
\(T_{G,K}=T_{H,K}\circ T_{G,H}\)
of Artin transfers, and consequently
\(T_{G,K}(x)=T_{H,K}(T_{G,H}(x))=T_{H,K}(1)=1\)
for all \(x\in\ker(T_{G,H})\),
that is \(\ker(T_{G,H})\le\ker(T_{G,K})\).
Further, \(\ker(T_{G,G})=G^\prime\) and \(\ker(T_{G,G^\prime})=G\).
\end{proof}


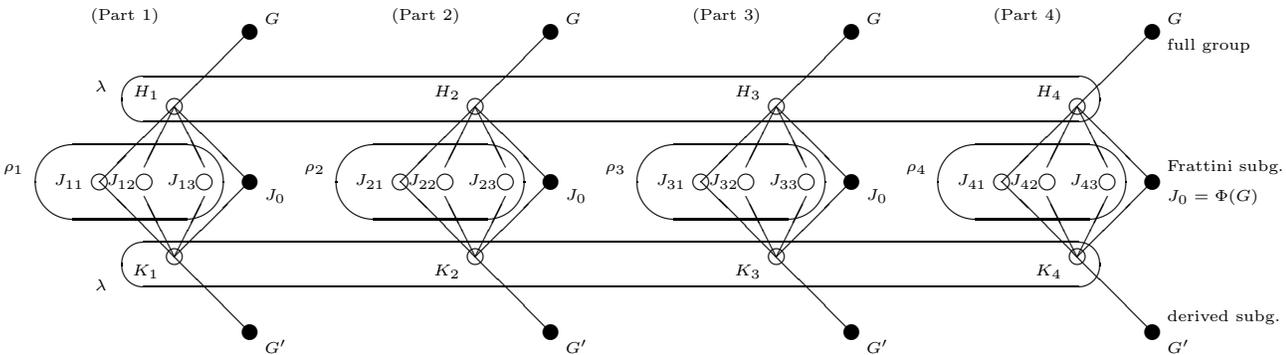
\begin{figure}[ht]
\caption{Permutations of harmonically balanced capitulation}
\label{fig:HarmonicallyBalanced}

{\tiny

\setlength{\unitlength}{1.0cm}
\begin{picture}(14,6)(-12,-10)



\multiput(-6.2,-5.9)(0,-2.2){2}{\oval(13,0.6)}
\put(-12.9,-5.8){\makebox(0,0)[rb]{\(\lambda\)}}
\put(-12.9,-8.3){\makebox(0,0)[rt]{\(\lambda\)}}
\multiput(-0.6,-7)(-4,0){4}{\oval(2.5,1)}
\put(-14,-6.9){\makebox(0,0)[rb]{\(\rho_1\)}}
\put(-10,-6.9){\makebox(0,0)[rb]{\(\rho_2\)}}
\put(-6,-6.9){\makebox(0,0)[rb]{\(\rho_3\)}}
\put(-2,-6.9){\makebox(0,0)[rb]{\(\rho_4\)}}



\put(-12.2,-4.9){\makebox(0,0)[rb]{(Part \(1\))}}
\put(-11,-5){\circle*{0.2}}
\put(-10.8,-4.9){\makebox(0,0)[lb]{\(G\)}}
\put(-10.8,-5.1){\makebox(0,0)[lt]{}}

\put(-11,-5){\line(-1,-1){2}}

\put(-12,-6){\circle{0.2}}
\put(-12.2,-5.9){\makebox(0,0)[rb]{\(H_1\)}}

\put(-12,-6){\line(-1,-2){0.4}}
\put(-12,-6){\line(1,-2){0.4}}
\put(-12,-6){\line(1,-1){1}}

\put(-13,-7){\circle{0.2}}
\put(-13.2,-7){\makebox(0,0)[rc]{\(J_{11}\)}}
\put(-12.4,-7){\circle{0.2}}
\put(-12.5,-7){\makebox(0,0)[rc]{\(J_{12}\)}}
\put(-11.6,-7){\circle{0.2}}
\put(-11.7,-7){\makebox(0,0)[rc]{\(J_{13}\)}}

\put(-10.8,-6.9){\makebox(0,0)[lb]{}}
\put(-11,-7){\circle*{0.2}}
\put(-10.8,-7.1){\makebox(0,0)[lt]{\(J_0\)}}

\put(-12,-8){\line(-1,2){0.4}}
\put(-12,-8){\line(1,2){0.4}}
\put(-12,-8){\line(1,1){1}}

\put(-12,-8){\circle{0.2}}
\put(-12.2,-8.1){\makebox(0,0)[rt]{\(K_1\)}}

\put(-11,-9){\line(-1,1){2}}

\put(-11,-9){\circle*{0.2}}
\put(-10.8,-8.9){\makebox(0,0)[lb]{}}
\put(-10.8,-9.1){\makebox(0,0)[lt]{\(G^\prime\)}}



\put(-8.2,-4.9){\makebox(0,0)[rb]{(Part \(2\))}}
\put(-7,-5){\circle*{0.2}}
\put(-6.8,-4.9){\makebox(0,0)[lb]{\(G\)}}
\put(-6.8,-5.1){\makebox(0,0)[lt]{}}

\put(-7,-5){\line(-1,-1){2}}

\put(-8,-6){\circle{0.2}}
\put(-8.2,-5.9){\makebox(0,0)[rb]{\(H_2\)}}

\put(-8,-6){\line(-1,-2){0.4}}
\put(-8,-6){\line(1,-2){0.4}}
\put(-8,-6){\line(1,-1){1}}

\put(-9,-7){\circle{0.2}}
\put(-9.2,-7){\makebox(0,0)[rc]{\(J_{21}\)}}
\put(-8.4,-7){\circle{0.2}}
\put(-8.5,-7){\makebox(0,0)[rc]{\(J_{22}\)}}
\put(-7.6,-7){\circle{0.2}}
\put(-7.7,-7){\makebox(0,0)[rc]{\(J_{23}\)}}

\put(-6.8,-6.9){\makebox(0,0)[lb]{}}
\put(-7,-7){\circle*{0.2}}
\put(-6.8,-7.1){\makebox(0,0)[lt]{\(J_0\)}}

\put(-8,-8){\line(-1,2){0.4}}
\put(-8,-8){\line(1,2){0.4}}
\put(-8,-8){\line(1,1){1}}

\put(-8,-8){\circle{0.2}}
\put(-8.2,-8.1){\makebox(0,0)[rt]{\(K_2\)}}

\put(-7,-9){\line(-1,1){2}}

\put(-7,-9){\circle*{0.2}}
\put(-6.8,-8.9){\makebox(0,0)[lb]{}}
\put(-6.8,-9.1){\makebox(0,0)[lt]{\(G^\prime\)}}



\put(-4.2,-4.9){\makebox(0,0)[rb]{(Part \(3\))}}
\put(-3,-5){\circle*{0.2}}
\put(-2.8,-4.9){\makebox(0,0)[lb]{\(G\)}}
\put(-2.8,-5.1){\makebox(0,0)[lt]{}}

\put(-3,-5){\line(-1,-1){2}}

\put(-4,-6){\circle{0.2}}
\put(-4.2,-5.9){\makebox(0,0)[rb]{\(H_3\)}}

\put(-4,-6){\line(-1,-2){0.4}}
\put(-4,-6){\line(1,-2){0.4}}
\put(-4,-6){\line(1,-1){1}}

\put(-5,-7){\circle{0.2}}
\put(-5.2,-7){\makebox(0,0)[rc]{\(J_{31}\)}}
\put(-4.4,-7){\circle{0.2}}
\put(-4.5,-7){\makebox(0,0)[rc]{\(J_{32}\)}}
\put(-3.6,-7){\circle{0.2}}
\put(-3.7,-7){\makebox(0,0)[rc]{\(J_{33}\)}}

\put(-2.8,-6.9){\makebox(0,0)[lb]{}}
\put(-3,-7){\circle*{0.2}}
\put(-2.8,-7.1){\makebox(0,0)[lt]{\(J_0\)}}

\put(-4,-8){\line(-1,2){0.4}}
\put(-4,-8){\line(1,2){0.4}}
\put(-4,-8){\line(1,1){1}}

\put(-4,-8){\circle{0.2}}
\put(-4.2,-8.1){\makebox(0,0)[rt]{\(K_3\)}}

\put(-3,-9){\line(-1,1){2}}

\put(-3,-9){\circle*{0.2}}
\put(-2.8,-8.9){\makebox(0,0)[lb]{}}
\put(-2.8,-9.1){\makebox(0,0)[lt]{\(G^\prime\)}}



\put(-0.2,-4.9){\makebox(0,0)[rb]{(Part \(4\))}}
\put(1,-5){\circle*{0.2}}
\put(1.2,-4.9){\makebox(0,0)[lb]{\(G\)}}
\put(1.2,-5.1){\makebox(0,0)[lt]{full group}}

\put(1,-5){\line(-1,-1){2}}

\put(0,-6){\circle{0.2}}
\put(-0.2,-5.9){\makebox(0,0)[rb]{\(H_4\)}}

\put(0,-6){\line(-1,-2){0.4}}
\put(0,-6){\line(1,-2){0.4}}
\put(0,-6){\line(1,-1){1}}

\put(-1,-7){\circle{0.2}}
\put(-1.2,-7){\makebox(0,0)[rc]{\(J_{41}\)}}
\put(-0.4,-7){\circle{0.2}}
\put(-0.5,-7){\makebox(0,0)[rc]{\(J_{42}\)}}
\put(0.4,-7){\circle{0.2}}
\put(0.3,-7){\makebox(0,0)[rc]{\(J_{43}\)}}

\put(1.2,-6.9){\makebox(0,0)[lb]{Frattini subg.}}
\put(1,-7){\circle*{0.2}}
\put(1.2,-7.1){\makebox(0,0)[lt]{\(J_0=\Phi(G)\)}}

\put(0,-8){\line(-1,2){0.4}}
\put(0,-8){\line(1,2){0.4}}
\put(0,-8){\line(1,1){1}}

\put(0,-8){\circle{0.2}}
\put(-0.2,-8.1){\makebox(0,0)[rt]{\(K_4\)}}

\put(1,-9){\line(-1,1){2}}

\put(1,-9){\circle*{0.2}}
\put(1.2,-8.9){\makebox(0,0)[lb]{derived subg.}}
\put(1.2,-9.1){\makebox(0,0)[lt]{\(G^\prime\)}}


\end{picture}

}

\end{figure}


\begin{remark}
\label{rmk:Scenarios}
For imaginary quadratic base fields, three principal scenarios of capitulation seem to occur.
The transfer kernels are either
\textit{bi-polarized}, e.g. \(\varkappa_1=(K_4,K_1,K_1,K_1)\), and thus necessarily \(\ker(T_{G,J_0})=J_0\), or
\textit{uni-polarized}, e.g. \(\varkappa_1=(K_1,K_1,K_1,K_1)\) with exceptional \(\ker(T_{G,J_0})=J_{11}\), or
\textit{harmonically balanced} (the most frequent scenario), as described by the following proposition.
\end{remark}


\begin{proposition}
\label{prp:Harmonic}
\(G\) possesses harmonically balanced capitulation (briefly \(\mathrm{HBC}\)) if and only if
\(\ker(T_{G,J_0})=J_0\)
and there exists a permutation \(\lambda\in\mathfrak{S}_4\)
such that for all \(1\le i\le 4\)
\begin{enumerate}
\item
\(\ker(T_{G,H_i})=K_{\lambda(i)}\), 
\item
\((\exists\rho_i\in\mathfrak{S}_3)\) \((\forall 1\le k\le 3)\) \(\ker(T_{G,J_{ik}})=J_{\lambda(i),\rho_i(k)}\),
\item
\(\ker(T_{G,K_i})=H_{\lambda(i)}\).
\end{enumerate}
\end{proposition}

\begin{proof}
This is a consequence of Definition
\ref{dfn:Harmonic}
under the constraints of Proposition
\ref{prp:Inclusion}.
\end{proof}

\noindent
The five permutations involved in harmonically balanced capitulation are
illustrated in Figure
\ref{fig:HarmonicallyBalanced}.


\section{Imaginary quadratic fields of type \((9,9)\)}
\label{s:Imaginary}
\noindent
Since Scholz and Taussky
\cite{SoTa}
were generally interested in the way
how more and more ideal classes of order a power of \(p\) successively become principal
in the subfields \(N\) of the Hilbert \(p\)-class field \(\mathrm{F}_p^1(K)\)
of a base field \(K\) with non-trivial \(p\)-class group,
it might well be true that the genius Scholz intended to investigate
\(3\)-class field towers of imaginary quadratic fields \(K\)
with \(3\)-class group \(\mathrm{Cl}_3(K)\simeq C_9\times C_9\),
because he had a presentiment
that the capitulation will frequently be \textit{harmonically balanced}.
In the sequel, we show that this is indeed the case.


\begin{theorem}
\label{thm:MetabelianI1} (Two-stage tower 1.)
Let \(K\) be an imaginary quadratic field
with Artin pattern
\begin{equation}
\label{eqn:APMetabelianI1}
\begin{aligned}
\tau_0&=(22), \quad \tau_1=((221)^4), \quad \varkappa_1=(K_1,K_2,K_3,K_4), \quad \tau_2=((321)^{12};222), \\
\varkappa_2&=(J_{13},J_{11},J_{12},J_{21},J_{22},J_{23},J_{31},J_{32},J_{33},J_{41},J_{42},J_{43};J_0), \text{ resp.} \\
\varkappa_2&=(J_{13},J_{11},J_{12},J_{23},J_{21},J_{22},J_{33},J_{31},J_{32},J_{42},J_{43},J_{41};J_0).
\end{aligned}
\end{equation}
Then the \(3\)-class field tower of \(K\) has length \(\ell_3(K)=2\)
and Galois group \(\mathfrak{G}=\mathrm{Gal}(\mathrm{F}_3^\infty(K)/K)\)
isomorphic to one of the two metabelian Schur \(\sigma2\)-groups
\(\langle 6561,28\rangle-\#2;i\) with \(i=4\), resp. \(i=5\), coclass \(7\),
\(\tau_3=((222)^4)\), \(\varkappa_3=(H_1,H_2,H_3,H_4)\) and \(\tau_4=(222)\).
Both groups possess harmonically balanced capitulation:
\(\lambda=1\), \(\rho_1=(132)\), \(\rho_2=\rho_3=\rho_4=1\), resp. \(\rho_2=\rho_3=(132)\), \(\rho_4=(123)\).
\end{theorem}


\begin{example}
\label{exm:MetabelianI1}
The information in Formula
\eqref{eqn:APMetabelianI1}
occurs rather sparsely as Artin pattern
of imaginary quadratic fields \(K\) with \(3\)-class group of type \((9,9)\).
The absolutely smallest discriminant \(d_K\) with this pattern is
\(-426\,291=-3\cdot 142\,097\).
Since the ordering of the members of \(\varkappa_2\)
cannot be determined strictly by arithmetic computations,
it is impossible to decide whether \(i=4\) or \(i=5\).
\end{example}


\begin{proof}
(Proof of Theorem
\ref{thm:MetabelianI1})
The descendant tree of the metabelian but non-abelian root \(\langle 243,2\rangle\)
with depth \(2\) and step size up to \(3\)
is constructed by means of the \(p\)-group generation algorithm
\cite{Nm,Ob},
which is implemented in Magma
\cite{BCP,BCFS,MAGMA}.
The resulting \(27\,222\) vertices are sifted by pattern recognition with respect to
the transfer targets \(\tau_0=(22)\), \(\tau_1=((221)^4)\), \(\tau_2=((321)^{12};222)\),
and the relation rank \(d_2=2\) of the required balanced presentation,
but without taking into consideration the transfer kernels \(\varkappa_1\) and \(\varkappa_2\).
This filtering process unambiguously leads to
the two metabelian Schur \(\sigma2\)-groups
\(\langle 6561,28\rangle-\#2;i\) with \(i=4\) resp. \(i=5\),
which also satisfy the requirements of harmonically balanced capitulation.
So the assumptions concerning \(\varkappa_1\) and \(\varkappa_2\) could be removed from Theorem
\ref{thm:MetabelianI1}.
The other three siblings are also metabelian Schur \(\sigma2\)-groups
but they possess different second layer Artin patterns: \\
\(\tau_2=((222)^3,(321)^6,(222)^3;222)\) for \(i=1\), \\
\(\tau_2=((222)^3,(321)^3,(222)^6;222)\) for \(i=2\), \\
\(\tau_2=((222)^3,(321)^9;222)\) for \(i=3\). \\
This additional information could be useful for studying other base fields of type \((9,9)\).
\end{proof}


\begin{theorem}
\label{thm:MetabelianI2a} (Two-stage tower 2a.)
Let \(K\) be an imaginary quadratic field
with Artin pattern
\begin{equation}
\label{eqn:APMetabelianI2a}
\begin{aligned}
\tau_0&=(22), \quad \tau_1=((311)^2,221,311), \quad \varkappa_1=(K_4,K_1,K_3,K_2), \\
\tau_2&=((42)^6,(321)^3,(42)^3;222), \\
\varkappa_2&=(J_{41},J_{43},J_{42},J_{11},J_{12},J_{13},J_{31},J_{32},J_{33},J_{23},J_{22},J_{21};J_0), \text{ resp.} \\
\varkappa_2&=(J_{41},J_{43},J_{42},J_{12},J_{13},J_{11},J_{32},J_{33},J_{31},J_{21},J_{23},J_{22};J_0).
\end{aligned}
\end{equation}
Then the \(3\)-class field tower of \(K\) has length \(\ell_3(K)=2\)
and Galois group \(\mathfrak{G}=\mathrm{Gal}(\mathrm{F}_3^\infty(K)/K)\)
isomorphic to one of the two metabelian Schur \(\sigma\)-groups
\(\langle 6561,26\rangle-\#2;i\) with \(i=1\) resp. \(i=3\), coclass \(7\),
\(\tau_3=((321)^2),222,321\), \(\varkappa_3=(H_4,H_1,H_3,H_2)\) and \(\tau_4=(222)\).
Both groups possess \(\mathrm{HBC}\):
\(\lambda=(142)\), \(\rho_1=(23)\), \(\rho_2=\rho_3=1\), \(\rho_4=(13)\), resp. \(\rho_2=\rho_3=(123)\), \(\rho_4=(23)\).
\end{theorem}


\begin{theorem}
\label{thm:MetabelianI2b} (Two-stage tower 2b.)
Let \(K\) be an imaginary quadratic field
with Artin pattern
\begin{equation}
\label{eqn:APMetabelianI2b}
\begin{aligned}
\tau_0&=(22), \quad \tau_1=(311,221,(311)^2), \quad \varkappa_1=(K_4,K_2,K_1,K_3), \\
\tau_2&=((42)^3,(321)^3,(42)^6;222), \\
\varkappa_2&=(J_{41},J_{43},J_{42},J_{21},J_{22},J_{23},J_{11},J_{12},J_{13},J_{31},J_{33},J_{32};J_0), \text{ resp.} \\
\varkappa_2&=(J_{41},J_{43},J_{42},J_{22},J_{23},J_{21},J_{12},J_{13},J_{11},J_{32},J_{31},J_{33};J_0).
\end{aligned}
\end{equation}
Then the \(3\)-class field tower of \(K\) has length \(\ell_3(K)=2\)
and Galois group \(\mathfrak{G}=\mathrm{Gal}(\mathrm{F}_3^\infty(K)/K)\)
isomorphic to one of the two metabelian Schur \(\sigma\)-groups
\(\langle 6561,27\rangle-\#2;i\) with \(i=1\) resp. \(i=2\), coclass \(7\),
\(\tau_3=(321,222,(321)^2)\), \(\varkappa_3=(H_4,H_2,H_1,H_3)\) and \(\tau_4=(222)\).
Both groups possess \(\mathrm{HBC}\):
\(\lambda=(143)\), \(\rho_1=(23)\), \(\rho_2=\rho_3=1\), \(\rho_4=(23)\), resp. \(\rho_2=\rho_3=(123)\), \(\rho_4=(12)\).
\end{theorem}


\begin{proof}
(Simultaneous proof of Theorem
\ref{thm:MetabelianI2a}
and Theorem
\ref{thm:MetabelianI2b})
Again we start at the metabelian but non-abelian root \(\langle 243,2\rangle\)
and construct its descendant tree
with depth \(2\) and step size up to \(3\)
by means of the \(p\)-group generation algorithm.
Again the result consists of \(27\,222\) vertices
which are sifted by pattern recognition with respect to
the occupation numbers of the unordered transfer targets
\(\tau_0=(22)\), \(\tau_1=((311)^3,221)\), \(\tau_2=((42)^9,(321)^3;222)\),
which coincide for
\eqref{eqn:APMetabelianI2a}
and
\eqref{eqn:APMetabelianI2b},
and the relation rank \(d_2=2\) of the required balanced presentation,
but without taking into consideration the transfer kernels \(\varkappa_1\) and \(\varkappa_2\).
This filtering process leads to precisely
four metabelian Schur \(\sigma2\)-groups,
\(\langle 6561,26\rangle-\#2;i\) with \(i=1\) resp. \(i=3\), and
\(\langle 6561,27\rangle-\#2;i\) with \(i=1\) resp. \(i=2\).
They also satisfy the requirements of harmonically balanced capitulation,
so again the assumptions concerning \(\varkappa_1\) and \(\varkappa_2\) could be removed from Theorems
\ref{thm:MetabelianI2a}
and
\ref{thm:MetabelianI2b}.
For possible later use we note that
the other sibling
\(\langle 6561,26\rangle-\#2;2\),
resp.
\(\langle 6561,27\rangle-\#2;3\),
is also a metabelian Schur \(\sigma2\)-group
with different second layer Artin pattern: 
\(\tau_2=((42)^6,(222)^3,(42)^3;222)\), resp.
\(\tau_2=((42)^3,(222)^3,(42)^6;222)\).
\end{proof}


\begin{example}
\label{exm:MetabelianI2}
The information in Formula
\eqref{eqn:APMetabelianI2a}
resp.
\eqref{eqn:APMetabelianI2b}
occurs frequently as Artin pattern
of imaginary quadratic fields \(K\) with \(3\)-class group of type \((9,9)\).
The absolutely smallest five discriminants \(d_K\) with these patterns are
\(-647\,359\) prime, \(-654\,468=-2^2\cdot 3\cdot 54\,539\), \(-682\,811\) prime,
\(-1\,333\,159=-29\cdot 45\,971\) and \(-1\,443\,567=-3\cdot 79\cdot 6091\).
Since the terms of \((\varkappa_1,\tau_1)\) and \((\varkappa_2,\tau_2)\)
are produced in random ordering
by the class field theoretic routines of Magma
\cite{MAGMA},
it is not possible to decide which of the four contestant groups in Theorems
\ref{thm:MetabelianI2a}
and
\ref{thm:MetabelianI2b}
actually occurs as \(\mathfrak{G}=\mathrm{Gal}(\mathrm{F}_3^\infty(K)/K)\).
\end{example}


\begin{theorem}
\label{thm:MetabelianI3a} (Two-stage tower 3a.)
Let \(K\) be an imaginary quadratic field
with Artin pattern
\begin{equation}
\label{eqn:APMetabelianI3a}
\begin{aligned}
\tau_0&=(22), \quad \tau_1=((311)^2,221,311), \quad \varkappa_1=(K_4,K_1,K_3,K_2), \\
\tau_2&=((42)^6,(222)^3,(42)^3;222), \\
\varkappa_2&=(J_{41},J_{43},J_{42},J_{13},J_{11},J_{12},J_{33},J_{31},J_{32},J_{22},J_{21},J_{23};J_0).
\end{aligned}
\end{equation}
Then the \(3\)-class field tower of \(K\) has length \(\ell_3(K)=2\)
and Galois group \(\mathfrak{G}=\mathrm{Gal}(\mathrm{F}_3^\infty(K)/K)\)
isomorphic to the unique metabelian Schur \(\sigma\)-group
\(\langle 6561,26\rangle-\#2;2\) with
\(\tau_3=((321)^2,222,321)\), \(\varkappa_3=(H_4,H_1,H_3,H_2)\),\(\tau_4=(222)\) and coclass \(7\).
The group possesses harmonically balanced capitulation:
\(\lambda=(142)\), \(\rho_1=(23)\), \(\rho_2=\rho_3=(132)\), \(\rho_4=(12)\).
\end{theorem}


\begin{theorem}
\label{thm:MetabelianI3b} (Two-stage tower 3b.)
Let \(K\) be an imaginary quadratic field
with Artin pattern
\begin{equation}
\label{eqn:APMetabelianI3b}
\begin{aligned}
\tau_0&=(22), \quad \tau_1=(311,221,(311)^2), \quad \varkappa_1=(K_4,K_2,K_1,K_3), \\
\tau_2&=((42)^3,(222)^3,(42)^6;222), \\
\varkappa_2&=(J_{41},J_{43},J_{42},J_{23},J_{21},J_{22},J_{13},J_{11},J_{12},J_{33},J_{32},J_{31};J_0).
\end{aligned}
\end{equation}
Then the \(3\)-class field tower of \(K\) has length \(\ell_3(K)=2\)
and Galois group \(\mathfrak{G}=\mathrm{Gal}(\mathrm{F}_3^\infty(K)/K)\)
isomorphic to the unique metabelian Schur \(\sigma\)-group
\(\langle 6561,27\rangle-\#2;3\) with
\(\tau_3=(321,222,(321)^2)\), \(\varkappa_3=(H_4,H_2,H_1,H_3)\), \(\tau_4=(222)\) and  coclass \(7\).
The group possesses harmonically balanced capitulation:
\(\lambda=(143)\), \(\rho_1=(23)\), \(\rho_2=\rho_3=(132)\), \(\rho_4=(13)\).
\end{theorem}


\begin{proof}
(Simultaneous proof of Theorem
\ref{thm:MetabelianI3a}
and Theorem
\ref{thm:MetabelianI3b})
The claims are consequences of the remarks at the end of the
simultaneous proof of Theorem
\ref{thm:MetabelianI2a}
and Theorem
\ref{thm:MetabelianI2b}.
\end{proof}


\begin{example}
\label{exm:MetabelianI3}
The information in Formula
\eqref{eqn:APMetabelianI3a}
resp.
\eqref{eqn:APMetabelianI3b}
occurs rather sparsely as Artin pattern
of imaginary quadratic fields \(K\) with \(3\)-class group of type \((9,9)\).
The absolutely smallest discriminant \(d_K\) with this pattern is
\(-1\,287\,544=-2^3\cdot 227\cdot 709\).
Since the terms of \((\varkappa_1,\tau_1)\) and \((\varkappa_2,\tau_2)\)
are produced in random ordering
by the class field theoretic routines of Magma
\cite{MAGMA},
it is not possible to decide which of the two contestant groups in Theorems
\ref{thm:MetabelianI3a}
and
\ref{thm:MetabelianI3b}
actually occurs as \(\mathfrak{G}=\mathrm{Gal}(\mathrm{F}_3^\infty(K)/K)\).
\end{example}


\noindent
All the preceding theorems in this section
have caused the impression that harmonically balanced capitulation
is connected with two-stage towers of \(3\)-class fields, i.e. metabelian towers.
Figure
\ref{fig:RootRegion}
shows that the involved \(3\)-class tower groups \(\mathfrak{G}\)
are descendants of the finitely capable groups \(\langle 6561,i\rangle\) with \(26\le i\le 28\).
The following theorems, however,
will supplement the belated motivation
why we thoroughly prepared the study of \textit{tree topologies}
for the description of non-metabelian towers in \S\
\ref{s:Multifurcation}.
Now the descendants of the groups \(\langle 6561,i\rangle\) with \(23\le i\le 25\),
which also possess harmonically balanced capitulation,
will come into the play.
Except for \(i=24\), they are infinitely capable
and give rise to \textit{periodic trifurcations} and \textit{parametrized sequences} of Schur \(\sigma\)-groups,
as illustrated in Figure
\ref{fig:Trifurcations}.


\begin{theorem}
\label{thm:NonMetabelianI1} (Three-stage tower 1.)
Let \(K\) be an imaginary quadratic field
with Artin pattern
\begin{equation}
\label{eqn:APNonMetabelianI1}
\begin{aligned}
\tau_0&=(22), \quad \tau_1=((311)^4), \quad \varkappa_1=(K_4,K_3,K_2,K_1), \\
\tau_2&=((42)^{12};222), \\
\varkappa_2&=(J_{41},J_{43},J_{42},J_{31},J_{32},J_{33},J_{22},J_{23},J_{21},J_{13},J_{12},J_{11};J_0).
\end{aligned}
\end{equation}
Then the \(3\)-class field tower of \(K\) has length \(\ell_3(K)\ge 3\)
and Galois group \(\mathfrak{G}=\mathrm{Gal}(\mathrm{F}_3^\infty(K)/K)\)
isomorphic to one of the infinitely many non-metabelian Schur \(\sigma\)-groups
\(T_{i,k}\) with \(i\ge 1\), \(1\le k\le 3\),
in the descendant tree of \(V_0=\langle 6561,25\rangle\) with
fixed \(\tau_3=((331)^4)\), \(\varkappa_3=(H_4,H_3,H_2,H_1)\), 
and variable \(\tau_4=(2+i,2,2)\) and coclass \(5+3i\).
The metabelianization \(M=\mathfrak{G}/\mathfrak{G}^{\prime\prime}\) of \(\mathfrak{G}\)
is isomorphic to \(S_1=V_0-\#3;6\), if \(i=1\),
and to \(V_1-\#1;3\) with \(V_1=V_0-\#3;3\), if \(i\ge 2\).
All groups possess harmonically balanced capitulation:
\(\lambda=(14)(23)\), \(\rho_1=(23)\), \(\rho_2=1\), \(\rho_3=(123)\), \(\rho_4=(13)\).
\end{theorem}


\begin{theorem}
\label{thm:MetabelianI4} (Two-stage tower 4.)
Let \(K\) be an imaginary quadratic field
with Artin pattern
\begin{equation}
\label{eqn:APMetabelianI4}
\begin{aligned}
\tau_0&=(22), \quad \tau_1=((311)^4), \quad \varkappa_1=(K_4,K_1,K_2,K_3), \\
\tau_2&=((42)^{12};222), \\
\varkappa_2&=(J_{41},J_{42},J_{43},J_{11},J_{13},J_{12},J_{22},J_{21},J_{23},J_{31},J_{32},J_{33};J_0).
\end{aligned}
\end{equation}
Then the \(3\)-class field tower of \(K\) has length \(\ell_3(K)=2\)
and Galois group \(\mathfrak{G}=\mathrm{Gal}(\mathrm{F}_3^\infty(K)/K)\)
isomorphic to one of three metabelian Schur \(\sigma\)-groups
\(\langle 6561,24\rangle-\#2;k\) with \(1\le k\le 3\),
\(\tau_3=((321)^4)\), \(\varkappa_3=(H_4,H_1,H_2,H_3)\), 
\(\tau_4=(222)\) and coclass \(7\).
The groups possess harmonically balanced capitulation:
\(\lambda=(1432)\), \(\rho_1=\rho_4=1\), \(\rho_2=(23)\), \(\rho_3=(12)\).
\end{theorem}


\begin{example}
\label{exm:MetabelianOrNonMetabelian}
The information in Formula
\eqref{eqn:APNonMetabelianI1}
resp.
\eqref{eqn:APMetabelianI4}
occurs frequently as Artin pattern
of imaginary quadratic fields \(K\) with \(3\)-class group of type \((9,9)\).
The absolutely smallest seven discriminants \(d_K\) with these patterns are
\(-134\,059\) prime, \(-298\,483\) prime, \(-430\,411\) prime,
\(-660\,356=-2^2\cdot 165\,089\), \(-1\,231\,060=-2^2\cdot5\cdot 61\,553\),
\(-1\,332\,644=-2^2\cdot 333\,161\) and \(-1\,430\,867=-31\cdot 101\cdot 457\).
Since the terms of \((\varkappa_1,\tau_1)\) and \((\varkappa_2,\tau_2)\)
are produced in random ordering
by the class field theoretic routines of Magma
\cite{MAGMA},
it is not possible to decide whether one of the infinitely many candidates in Theorem
\ref{thm:NonMetabelianI1}
or one of the three contestants in Theorem
\ref{thm:MetabelianI4}
actually occurs as \(\mathfrak{G}=\mathrm{Gal}(\mathrm{F}_3^\infty(K)/K)\),
except one uses the following work around with huge computational complexity.
\end{example}


\noindent
According to
\cite[Thm. 5.4, pp. 86--87]{Ma2016a},
all higher \(p\)-class groups \(\mathrm{Gal}(\mathrm{F}_p^n(K)/K)\) with \(n\ge 2\)
of a number field \(K\)
share a common Artin pattern \(\mathrm{AP}(K)=(\tau(K),\varkappa(K))\),
and consequently this collection of invariants
is not able to distinguish between different lengths \(\ell_p(K)\ge 2\)
of the \(p\)-class field tower of \(K\).
In addition to the (partial) multi-layered Artin pattern
\((\tau_0;\tau_1,\varkappa_1;\tau_2,\varkappa_2)\),
which was used as the preamble of all preceding theorems,
\textit{higher abelian quotient invariants} are now required
\cite[Dfn. 1.18, p. 78]{Ma2016b}
for distinguishing between towers with two stages and towers with more stages.

\begin{definition}
\label{dfn:AQI2}
By the \textit{abelian quotient invariants of second order}
of a pro-\(3\) group \(G\) with abelianization \(G/G^\prime\simeq C_9\times C_9\)
we understand the family of transfer targets \(\tau_{2,1}:=\tau_{2,1}(G):=\)
\[
\bigl(
\tau(J_{11}),\tau(J_{12}),\tau(J_{13}),\tau(J_{21}),\tau(J_{22}),\tau(J_{23}),\tau(J_{31}),\tau(J_{32}),\tau(J_{33}),\tau(J_{41}),\tau(J_{42}),\tau(J_{43});
\tau(J_0)\bigr)
\]
with respect to the collections of Artin transfer homomorphisms \((T_{J,S}:\,J/J^\prime\to S/S^\prime)_{S\in\mathrm{Lyr}_1(J)}\)
for each \(J=J_{ik}\), \(1\le i\le 4\), \(1\le k\le 3\), and also for \(J=J_0\).
\end{definition}

\noindent
\(\tau_{2,1}\) includes the information \(\tau_3\) about the subgroups \(K_i\), \(1\le i\le 4\), with index \(27\) in \(G\),
but also data for subgroups \(S<G\) of index \(27\) which do not contain \(G^\prime\),
and which therefore correspond to unramified but non-abelian extensions of relative degree \(27\) of \(K\),
that is of absolute degree \(54\).
(See also
\cite[\S\ 7.1, Eqn. (7.1), p. 149]{Ma2017a} and
\cite[\S\ 3.2, Eqn. (3.4), p. 664]{Ma2017b}.)


\begin{theorem}
\label{thm:StageSeparation} (Stage separation criterion.)
Let \(K\) be an imaginary quadratic field
with Artin pattern
\eqref{eqn:APNonMetabelianI1}
resp.
\eqref{eqn:APMetabelianI4}.
Then the \(3\)-class field tower of \(K\) has length \(\ell_3(K)=2\)
and one of the three possible Galois groups
\(\mathfrak{G}=\mathrm{Gal}(\mathrm{F}_3^\infty(K)/K)\)
in Theorem
\ref{thm:MetabelianI4},
if and only if
\[
\tau_{2,1}=\bigl(\lbrack (411)^3,321\rbrack^{12};\lbrack (321)^{12},331\rbrack\bigr),
\]
but it has length \(\ell_3(K)\ge 3\) and one of the infinitely many possible Galois groups
\(\mathfrak{G}\)
in Theorem
\ref{thm:NonMetabelianI1},
if
\[
\tau_{2,1}=\bigl(\lbrack (511)^3,331\rbrack^{12};\lbrack (322)^4,(331)^9\rbrack\bigr).
\]
\end{theorem}


\noindent
We conclude this section with a situation, which unambiguously enforces a
\(3\)-class field tower with at least three stages, \(\ell_3(K)\ge 3\),
without the necessity of computing abelian quotient invariants of second order \(\tau_{2,1}\)
with high computational complexity.


\begin{theorem}
\label{thm:NonMetabelianI2} (Three-stage tower 2.)
Let \(K\) be an \textbf{imaginary or real} quadratic field
with \(\mathrm{AP}\)
\begin{equation}
\label{eqn:APNonMetabelianI2}
\begin{aligned}
\tau_0&=(22), \quad \tau_1=(311,221,221,311), \quad \varkappa_1=(K_4,K_2,K_3,K_1), \\
\tau_2&=((42)^3,222,(321)^2,321,222,321,(42)^3;222), \\
\varkappa_2&=(J_{42},J_{43},J_{41},J_{23},J_{22},J_{21},J_{31},J_{33},J_{32},J_{11},J_{12},J_{13};J_0).
\end{aligned}
\end{equation}
Then the \(3\)-class field tower of \(K\) has length \(\ell_3(K)\ge 3\)
and Galois group \(\mathfrak{G}=\mathrm{Gal}(\mathrm{F}_3^\infty(K)/K)\)
isomorphic to one of the infinitely many non-metabelian Schur \(\sigma\)-groups
\(T_{i,k}\) with \(i\ge 1\), \(1\le k\le 3\),
in the descendant tree of \(V_0=\langle 6561,23\rangle\) with
fixed \(\tau_3=(331,322,322,331)\), \(\varkappa_3=(H_4,H_2,H_3,H_1)\), 
and variable \(\tau_4=(2+i,2,2)\) and coclass \(5+3i\).
The metabelianization \(M=\mathfrak{G}/\mathfrak{G}^{\prime\prime}\) of \(\mathfrak{G}\)
is isomorphic to \(S_1=V_0-\#3;9\), if \(i=1\),
and to \(V_1-\#1;4\) with \(V_1=V_0-\#3;7\), if \(i\ge 2\).
All groups possess harmonically balanced capitulation:
\(\lambda=(14)\), \(\rho_1=(123)\), \(\rho_2=(13)\), \(\rho_3=(23)\), \(\rho_4=1\).
The metabelianization \(M\) has \(d_2=4\),
forbidden as \(\mathfrak{G}\) for \textbf{any} quadratic field \(K\).
\end{theorem}


\begin{example}
\label{exm:NonMetabelianI2}
The information in Formula
\eqref{eqn:APNonMetabelianI2}
occurs as Artin pattern
of two imaginary quadratic fields 
and a single real quadratic field \(K\) with \(3\)-class group of type \((9,9)\).
The discriminants \(d_K\) with this pattern are
\(-475\,355=-5\cdot 95\,071\),
\(-1\,201\,272=-2^3\cdot 3\cdot 50\,053\), and
\(8\,965\,093=19\cdot 471\,847\).
According to Theorem
\ref{thm:NonMetabelianI2},
these three fields possess a non-metabelian \(3\)-class field tower of length \(\ell_3(K)\ge 3\).
\end{example}


\begin{proof}
(Simultaneous proof of Theorem
\ref{thm:NonMetabelianI1},
Theorem
\ref{thm:StageSeparation},
Theorem
\ref{thm:NonMetabelianI2},
and Theorem
\ref{thm:MetabelianI4})
All claims about abelian quotient invariants of second order \(\tau_{2,1}\)
and about the non-metabelian Schur \(\sigma2\)-groups \(T_{i,k}\)
have been verified with the aid of Magma
\cite{MAGMA}
as described in the proof of Theorem
\ref{thm:Trifurcation}
and expressed in Conjecture
\ref{cnj:Trifurcation}.
The process of pattern recognition was carried out
by a similar database query as in the proof of Theorem
\ref{thm:MetabelianI1},
which here resulted in descendants of \(\langle 6561,i\rangle\)
with \(i\in\lbrace 23,24,25\rbrace\) instead of \(i=28\).
\end{proof}


\section{Abelian root}
\label{s:Abelian}
\noindent
The abelian root of finite \(3\)-groups \(G\) with derived quotient \(G/G^\prime\simeq C_9\times C_9\)
is the non-elementary bihomocyclic \(3\)-group \(A:=\langle 81,2\rangle\simeq C_9\times C_9\) of order \(81=3^4\).
We expected that the root \(A\) and its descendants will not be involved
in the investigation of imaginary quadratic fields.
Contrary to our assumptions, however, it also seems to be irrelevant
for real quadratic fields.
All quadratic fields \(K\) with \(3\)-class group \(\mathrm{Cl}_3(K)\simeq C_9\times C_9\)
have higher non-abelian \(3\)-class groups \(\mathrm{Gal}(\mathrm{F}_3^n(K)/K)\), \(n\ge 2\),
in the descendant tree of the non-abelian root \(R=\langle 243,2\rangle\) discussed in \S\
\ref{s:Metabelian}.


\section{Real quadratic fields of type \((9,9)\)}
\label{s:Real}
\noindent
Since they start with extremely huge discriminants,
only \textit{four} real quadratic fields \(K\) with \(3\)-class group \(\mathrm{Cl}_3(K)\simeq C_9\times C_9\)
are known up to now.
The discriminants of these fields are
\[d_K\in\lbrace 8\,739\,521, 8\,965\,093, 15\,057\,697, 16\,279\,457 \rbrace.\]
The construction of the unramified abelian \(3\)-extensions with relative degree \(9\),
which form a multiplet with \(13\) members (a tridecuplet),
consumes much more CPU time for real than for imaginary base fields.
Therefore it was hard to analyze the four discriminants completely.


\begin{theorem}
\label{thm:MetabelianR} (Two stage tower.)
Let \(K\) be a real quadratic field
with Artin pattern
\begin{equation}
\label{eqn:APMetabelianR}
\begin{aligned}
\tau_0&=(22), \quad \tau_1=((211)^3,311), \quad \varkappa_1=((J_0)^3,K_1), \\
\tau_2&=((22)^9,42,(32)^2;221), \quad \varkappa_2=(G^9,J_{11},(H_1)^2;H_1).
\end{aligned}
\end{equation}
Then the \(3\)-class field tower of \(K\) has length \(\ell_3(K)=2\)
and Galois group \(\mathfrak{G}=\mathrm{Gal}(\mathrm{F}_3^\infty(K)/K)\)
isomorphic to one of the two metabelian \(\sigma2\)-groups
\(\langle 6561,i\rangle\) with \(i=383\) resp. \(i=384\), coclass \(5\), relation rank \(d_2=3\),
\(\tau_3=(211,(22)^2,32)\), \(\varkappa_3=(G,G^2,H_1)\) and \(\tau_4=(22)\).
Both groups are far away from possessing harmonically balanced capitulation.
\end{theorem}


\begin{example}
\label{exm:MetabelianR}
The information in Formula
\eqref{eqn:APMetabelianR}
occurs as Artin pattern
of three real quadratic fields \(K\) with \(3\)-class group of type \((9,9)\).
The discriminants \(d_K\) with this pattern are
\(8\,739\,521=7\cdot 1\,248\,503\),
\(15\,057\,697=43\cdot 350\,179\), and
\(16\,279\,457=59\cdot 275\,923\).
According to Theorem
\ref{thm:MetabelianR},
these three fields possess a metabelian \(3\)-class field tower of length \(\ell_3(K)=2\).
\end{example}


\begin{proof}
(Proof of Theorem
\ref{thm:MetabelianR})
A similar construction as in the proof of Theorem
\ref{thm:MetabelianI1}
but with looser constraint \(d_2\le 3\) for the relation rank
yields precisely the two groups in the claim.
Here it is sufficient to consider step sizes up to \(s=2\) only, instead of \(s=3\),
whence only \(961\) groups are generated, instead of the tedious batch of \(27222\).
\end{proof}


\section{Acknowledgements}
\label{s:Thanks}

\noindent
The author gratefully acknowledges that his research was supported by the Austrian Science Fund (FWF):
project P 26008-N25.



\end{document}